\documentclass{scrartcl}

\usepackage[utf8]{inputenc}
\usepackage[T1]{fontenc}
\usepackage{lmodern}
\usepackage[english]{babel}
\usepackage{amsmath}

\usepackage[center]{titlesec}

\usepackage[T1]{fontenc}

\usepackage[arrow, matrix, curve]{xy}

\usepackage{graphicx}
\usepackage{amssymb}
\usepackage{amsthm}

\usepackage[nottoc]{tocbibind}

\usepackage[perpage]{footmisc}

\usepackage{enumerate}

\usepackage[colorlinks,
pdfpagelabels,
pdfstartview = FitH,
bookmarksopen = true,
bookmarksnumbered = true,
linkcolor = blue,
plainpages = false,
hypertexnames = false,
citecolor = black] {hyperref}

\newtheorem{lem}{Lemma}

\newtheorem{prop}[lem]{Proposition}

\newtheorem{bem}[lem]{Remark}
\newtheorem{theorem}[lem]{Theorem}

\begin{document}

\begin{center}
\textbf{\LARGE{WITT VECTOR RINGS AND QUOTIENTS OF MONOID ALGEBRAS}} 
\begin{verbatim}
\end{verbatim}
Sina Ghassemi-Tabar 
\end{center}
\begin{verbatim}
\end{verbatim}
\begin{abstract} \noindent
\textsc{Abstract.} In a previous paper Cuntz and Deninger introduced the ring $C(R)$ for a perfect $\mathbb{F}_p$-algebra $R$. The ring $C(R)$ is canonically isomorphic to the $p$-typical Witt ring $W(R)$. In fact there exist canonical isomorphisms $\alpha_n \colon \mathbb{Z}R/I^n \xrightarrow{\sim} W_n(R)$. In this paper we give explicit descriptions of the isomorphisms $\alpha_n$ for $n\geq 2$ if $p\geq n$.
\end{abstract}

\section{Introduction}
For a perfect $\mathbb{F}_p$-algebra $R$ consider the monoid algebra $\mathbb{Z}R$ where $R$ is viewed as a monoid under multiplication. In \cite{CD1} the ring $C(R)$ is constructed as the $I$-adic completion of $\mathbb{Z}R$ where $I$ is the kernel of the natural projection $\pi\colon\mathbb{Z}R\to R$. It turns out that $C(R)$ is a strict $p$-ring with $C(R)/pC(R)=R$ and therefore canonically isomorphic to the ring of $p$-typical Witt vectors of $R$. As an immediate consequence we have a unique isomorphism $\alpha_n\colon \mathbb{Z}R/I^n \xrightarrow{\sim} W_n(R)$ for every $n\geq 2$, c.f. \cite{CD1} Remark 6 and Corollary 7. In \cite{CD1} there is an explicit description of the isomorphism $\alpha_2\colon \mathbb{Z}R/I^2 \xrightarrow{\sim} W_2(R)$.
It was verified by using addition and multiplication on the truncated Witt ring $W_2(R)$ to prove that $\alpha_2$ is a homomorphism and to conclude that it has to be the unique isomorphism. We choose another approach to calculate the isomorphism $\alpha_n$ by using the inverse map $\beta_n \colon W_n(R) \xrightarrow{\sim} \mathbb{Z}R/I^n $ given by the formula 
\begin{equation}
\beta_n(r_0, r_1, \dots, r_{n-1})=\sum_{k=0}^{n-1}p^k[\phi^{-k}(r_k)] \bmod I^n \label{beta_n}
\end{equation}
where $\phi$ is the Frobenius automorphism on $R$, c.f. \cite{CD2} Corollary 6.5, \cite{Ser} II §5. For background on the classical theory of Witt vectors see \cite{Haz}, \cite{Rab}. I would like to thank C. Deninger for suggesting the topic of this note and for helpful discussions.

\section{Determination of the isomorphism $\alpha_n$}

Let $R$ be a perfect $\mathbb{F}_p$-algebra. Consider the map $\phi \colon \mathbb{Z}R \to \mathbb{Z}R$, $\sum n_r[r] \mapsto \sum n_r[r^p]$. We have $\phi(x)\equiv x^p \bmod p\mathbb{Z}R $ for $x\in \mathbb{Z}R$. Therefore we can introduce the ``arithmetic derivation''
\begin{align*}
 \delta\colon \mathbb{Z}R &\to \mathbb{Z}R \text{,} \quad x\mapsto\frac{1}{p}(\phi(x)-x^p) \text{.}
\end{align*}
We mention some immediate facts about $\delta$.
\begin{prop}[\cite{CD1} page 2]
For $x, y, x_1, \dots, x_n \in \mathbb{Z}R$ we have:
\begin{enumerate}[(i)]
\item
\begin{equation}
\delta(x+y) = \delta(x) + \delta(y) - \sum_{k=1}^{p-1} \frac{1}{p} \binom{p}{k}x^k y^{p-k} \label{binom}
\end{equation}
\item
\begin{equation}
\delta(xy) = \delta(x)\phi(y) + x^p\delta(y) \label{prod}
\end{equation}
\item
\begin{equation*}
\delta(x_1 \cdots x_n) = \sum_{k=1}^n x_1^p \cdots x_{k-1}^p\delta(x_k)\phi(x_{k+1}) \cdots \phi(x_n)
\end{equation*}
\item
\begin{align}
 \delta(x+y) \equiv \delta(x)+\delta(y)\bmod I^n \qquad \text{if }x \in I^n \text{ or }y \in I^n \label{a+b}
\end{align}
\item
\begin{equation}
 \delta(I^n) \subset I^{n-1} \qquad \text{for n $\ge$ 1} \label{teil}
\end{equation}
\end{enumerate}
\end{prop}
With these properties we are able to derive equations which will be useful for determining the map $\alpha_n$ above.

\begin {lem}
For $a, b, c \in \mathbb{Z}R$ we get:
\begin{enumerate}[(i)]
\item
\begin{align}
a \equiv b \bmod I^n \Rightarrow \delta(a) \equiv \delta(b) \bmod I^{n-1} \label{a=b}
\end{align}
\item
\begin{align}
a \equiv b+c \bmod I^n \text{ and } c \in I^{n-1} \Rightarrow \delta(a) \equiv \delta(b) + \delta(c) \bmod I^{n-1} \label{a=b+c}
\end{align}
\item
\begin{align}
ab \in I^n \Rightarrow \delta(a+b)\equiv\delta(a)+\delta(b) \bmod I^n \label{ab}
\end{align}
\item
\begin{align}
\delta(p) \equiv 1 \bmod I^{p-1} \label{1}
\end{align}
\item
\begin{align*}
\delta(pa)\equiv \phi(a) \bmod I^{p-1}  
\end{align*}
In particular
\begin{align}
\delta(pa)\equiv \phi(a) \bmod I^k \qquad \text{for all }  0<k<p.  \label{pa}
\end{align}
\end{enumerate}
\end{lem}

\begin {proof}
\begin{enumerate}[$(i)$]
\item
For $y\in I^n$ we have by using equations (\ref{a+b}) and (\ref{teil})  
\begin{equation*}
a=b+y \Rightarrow \delta(a)=\delta(b+y) \equiv \delta(b)+\delta(y) \bmod I^n \equiv \delta(b) \bmod I^{n-1}
\end{equation*}
since $\delta(y) \in I^{n-1}$ and $I^n \subseteq I^{n-1}$.
\item
For $c \in I^{n-1}$ we have by using equations (\ref{a+b}) and (\ref{a=b})
\begin{equation*}
a \equiv b+c \bmod I^n \stackrel{(\ref{a=b})}{\Rightarrow} \delta(a) \equiv \delta(b+c) \bmod I^{n-1} 
 \stackrel{(\ref{a+b})}{\equiv} \delta(b) + \delta(c) \bmod I^{n-1} \text{.}
\end{equation*}
\item
The assertion follows from equation (\ref{binom}) because for $1\leq k \leq p-1$ every binomial coefficient $\binom{p}{k}$ is divisible by $p$.
\item
\begin{align*}
\delta(p) &= \delta(p\cdot1) =\delta(p[1]) =[1]-p^{p-1}[1] \equiv [1] \bmod I^{p-1} \equiv 1 \bmod I^{p-1}
\end{align*}
\item
\begin{align*}
\delta(pa) &\stackrel{(\ref{prod})}{=} \delta(p)\phi(a) + p^p\delta(a)                         
\stackrel{}{\equiv} \delta(p)\phi(a) \bmod I^{p-1}  \stackrel{(\ref{1})}{\equiv} \phi(a) \bmod I^{p-1} 
\end{align*}
\end{enumerate}
\end{proof}

As already mentioned, our aim is to describe the isomorphisms $\alpha_n\colon \mathbb{Z}R/I^n\xrightarrow{\sim} W_n(R)$ for $n\in \mathbb{N}$ by explicit formulas. We begin with the case $n=2$ to clarify the method. \\
\linebreak
\textbf{Determining the isomorphism $\alpha_2$} \\																																														
\linebreak
We obtain an explicit formula for the isomorphism $\alpha_2\colon \mathbb{Z}R/I^2\xrightarrow{\sim} W_2(R)$ by using the inverse map $\beta_2$ and the arithmetic derivation $\delta$. Because of formula (\ref{beta_n}) we know that for every element  $x \in \mathbb{Z}R/I^2$ there exist uniquely determined elements $r_0, r_1 \in R$ with
$$
x= \beta_2(r_0, r_1)= [r_0]+ p[\phi^{-1}(r_1)]\bmod I^2 \qquad \text{and} \qquad \alpha_2(x)=(r_0, r_1)\text{.}
$$
Note here that $p\in I$ since $R$ is an $\mathbb{F}_p$-algebra. So by reducing the first equation modulo $I$ we obtain  $x \equiv [r_0] \bmod I$ in $\mathbb{Z}R$ and therefore $r_0 = \pi(x)$. Substituting $\pi(x)$ for $r_0$ in the above equation, we have the identity
\begin{align*}
&x \equiv [\pi(x)] + p[\phi^{-1}(r_1)] \bmod I^2\text{.} 
\end{align*}
The remaining component $r_1$ can now be determined by applying $\delta$. We define $a$ to be $a:=\phi^{-1}(r_1)$. By applying $\delta$ to the second term on the right-hand side we have
\begin{align*}
\delta(p[a]) 	= \frac{1}{p}(\phi(p[a]) - (p[a])^p) 
		=\frac{1}{p}(p[a^p] - p^p[a]^p) 
		= [a^p] - p^{p-1}[a]^p.
\end{align*}
Since $p-1\ge1$ and $\delta([\cdot])=0$ we obtain the following congruence:
\begin{align*} 
\delta(x) \stackrel{(\ref{a=b+c})}{\equiv} \delta([\pi(x)]) + \delta(p[a]) \bmod I  
                      \equiv \delta(p[a]) \bmod I 
		 \equiv [a^p] \bmod I 
		 \equiv[r_1] \bmod I
\end{align*}
So for any element $y \in I$ satisfying $\delta(x)=[r_1]+y$ we obtain
\begin{align*}
\pi(\delta(x)) 	= \pi([r_1] + y) 
                             	= \pi([r_1]) + \pi(y) 
			= \pi([r_1]) 
			= r_1
\end{align*}
since $I$ is the kernel of the natural projection $\pi$.
In conclusion, the isomorphism $\alpha_2$ is given by the formula
\begin{equation}
\alpha_2(x) = (\pi(x), \pi(\delta(x))\text{.} \label{alpha_2}
\end{equation} 
As mentioned earlier, \cite{CD1} Proposition 8 uses a different approach to obtain this formula. \\
\linebreak
\textbf{Determining the isomorphism $\alpha_3$} \\                                              																																									
\linebreak
As already described we have
$$x\equiv [r_0]+ p[\phi^{-1}(r_1)]+p^2[\phi^{-2}(r_2)]\bmod I^3$$
with uniquely determined elements $r_0, r_1, r_2 \in R$. By reducing modulo $I^2$ we obtain analogue results as above for $r_0$ and $r_1$. So we can focus on calculating the component $r_2$. Therefore we apply $\delta$ two times.
Using equation (\ref{a=b+c}) leads to the congruences
\begin{align*}
x-[r_0] &\equiv  p[\phi^{-1}(r_1)]+p^2[\phi^{-2}(r_2)]\bmod I^3 \\
(\ref{a=b+c}) \Rightarrow \delta(x-[r_0]) &\equiv \delta( p[\phi^{-1}(r_1)])+\delta(p^2[\phi^{-2}(r_2)]) \bmod I^2 \\
\Rightarrow \delta(x-[r_0]) &\equiv [r_1] -p^{p-1}[r_1]+p[\phi^{-1}(r_2)]- p^{2p-1}[\phi^{-1}(r_2)] \bmod I^2 \\
2p-1>2\Rightarrow \delta(x-[r_0]) &\equiv [r_1] -p^{p-1}[r_1]+p[\phi^{-1}(r_2)] \bmod I^2  \\
\Rightarrow \delta(x-[r_0]) - [r_1] + p^{p-1}[r_1] &\equiv p[\phi^{-1}(r_2)] \bmod I^2 \text{.} \tag{$\ast$}
\end{align*}
At this point we have two different cases depending on the prime number $p$.\\
\linebreak
\underline{1. case: $p\geq 3$ } \\
\linebreak
In this case it applies that $p^{p-1} \in I^2$ and we obtain from ($\ast$) the congruence
\begin {align*}
\delta(x-[r_0])-[r_1] \equiv p[\phi^{-1}(r_2)] \bmod I^2\text{.} 
\end{align*}
Applying $\delta$ once again and using equation (\ref{a=b}) we have
\begin {align*}
&\delta(\delta(x-[r_0])-[r_1]) \equiv \delta(p[\phi^{-1}(r_2)]) \bmod I \\
\Rightarrow & \delta(\delta(x-[r_0])-[r_1]) \equiv [r_2] - p^{p-1}[r_2] \bmod I \\
\Rightarrow & \delta(\delta(x-[r_0])-[r_1]) \equiv [r_2] \bmod I
\end{align*}
which means that 
\begin{equation*}
r_2=\pi( \delta(\delta(x-[r_0])-[r_1]))\text{.}
\end{equation*}
\underline{2. case: $p=2$ } \\
\linebreak
From ($\ast$) we obtain
\begin{align*}
\delta(x-[r_0])+[r_1] &\equiv 2[\phi^{-1}(r_2)] \bmod I^2 \\
(\ref{a=b}) \Rightarrow \delta(\delta(x-[r_0])+[r_1]) &\equiv [r_2]- 2[r_2] \bmod I \\
p=2\in I \Rightarrow \delta(\delta(x-[r_0])+[r_1]) &\equiv [r_2] \bmod I \text{.}
\end{align*}
So we have
\begin{equation*}
r_2=\pi( \delta(\delta(x-[r_0])+[r_1]))\text{.}
\end{equation*}
For all prime numbers $p$ the third component is given by
\begin {equation*}
r_2=\pi(\delta(\delta(x-[r_0])+(-1)^p[r_1]))
\end{equation*}
and the isomorphism $\alpha_3$ is determined by
\begin{align}
\alpha_3(x)=  (\pi(x), \pi(\delta(x)), \pi(\delta(\delta(x-[\pi(x)])+(-1)^p[\pi(\delta(x))]))) \text{.}
\end{align} 
The same method can be used to determine the isomorphism $\alpha_4$ for all prime numbers $p$. Already in this case the formulas for the small primes $p=2$ and $p=3$ are quite complicated. For general $n\geq 2$ we therefore concentrate on the primes $p\geq n$. Here is the main result:

\begin{theorem} \label{erg1}
For $n\geq2$ and $p\geq n$ the $\nu$-th Witt vector component $r_\nu$ of an element $x\in \mathbb{Z}R/I^n$ under the isomorphism $\alpha_n\colon \mathbb{Z}R/I^n \xrightarrow{\sim} W_n(R) = R^n$ is recursively given by
\begin{align*}
r_\nu = \pi(\delta(\cdots\delta(\delta(x-[r_0])-[r_1])\cdots-[r_{\nu-1}]) ) \qquad \text{for } \nu=1, \dots, n-1
\end{align*}
with the $0$-th component being $r_0=\pi(x)$.
\end{theorem}

\begin{proof}
By formula (\ref{beta_n}) the elements $r_0, \dots, r_{n-1} \in R$ are uniquely determined by the formula
\begin{align*}
x\equiv\sum_{k=0}^{n-1}p^k[\phi^{-k}(r_k)] \bmod I^n \qquad \Big(\in \mathbb{Z}R/ I^n\Big)\text{.}
\end{align*}
To calculate the components, we proceed inductively. In the following we use equations (\ref{a=b}) and (\ref{pa}) to calculate the $\nu$-th component  $r_\nu$ for $\nu=0, \dots, n-1$. By reducing modulo $I^{\nu+1}$ we obtain
$$
x\equiv\sum_{k=0}^\nu p^k[\phi^{-k}(r_k)] \bmod I^{\nu+1} \text{.}
$$
For $\nu=0$ we are done by applying $\pi$. Otherwise we continue as follows: \\
\linebreak
\underline{Step 1:} By subtracting the first term on the right-hand side we have
\begin{align*}
x-[r_0]\equiv\sum_{k=1}^\nu p^k[\phi^{-k}(r_k)] \bmod I^{\nu+1} \text{.}
\end{align*}
\underline{Step 2:} We now use equation (\ref{a=b}) to obtain
\begin{align*}
\delta(x-[r_0])\equiv\delta(\sum_{k=1}^\nu p^k[\phi^{-k}(r_k)]) \bmod I^\nu 
\equiv \delta(p\Big(\sum_{k=1}^\nu p^{k-1}[\phi^{-k}(r_k)]\Big)) \bmod I^\nu\text{.}
\end{align*}
\underline{Step 3:} Because of our assumption $p\geq n$ we can use equation (\ref{pa}) to obtain
\begin{align*}
\delta(x-[r_0])&\equiv \phi(\sum_{k=1}^\nu p^{k-1}[\phi^{-k}(r_k)]) \bmod I^\nu \\
&\equiv \sum_{k=1}^\nu p^{k-1}[\phi^{-(k-1)}(r_k)]) \bmod I^\nu \\
&\equiv \sum_{k=0}^{\nu-1}p^{k}[\phi^{-k}(r_{k+1})]) \bmod I^\nu\text{.}
\end{align*}
By repeating these three steps $(\nu-1)$-times we finally have
\begin{align*}
[r_{n-1}] \bmod I \equiv \delta(\cdots\delta(\delta(x-[r_0])-[r_1])\cdots-[r_{n-2}]) \text{.}
\end{align*}
This implies the assertion by using the natural projection $\pi$ and $I=ker(\pi)$.
\end{proof}

\begin{bem}
Comparing Theorem \ref{erg1} for $n=2$ and formula (\ref{alpha_2}) we get
$$
\pi(\delta(x))=\pi(\delta(x-[\pi(x)])) \text{.}
$$
This can also be seen directly.
\end{bem}

\bibliography{Literatur}
\textsc{Westfälische Wilhelms-Universität, Fachbereich Mathematik, \\ Einsteinstraße 62, 48149 Münster, Germany} \\
\linebreak
E-mail address: sina dot ghassemi-tabar at uni-muenster dot de

\end{document}